\documentclass[reqno,12pt]{amsart}
\usepackage{latexsym}
\usepackage{amsxtra}
\usepackage{verbatim}
\usepackage{color}
\usepackage{amsthm,amsmath,amsfonts,amssymb,amscd,mathrsfs,graphics}
\usepackage{paralist}
\usepackage{amssymb}
\usepackage{times,float}
\usepackage{appendix}
\usepackage{txfonts}
\usepackage{hyperref,supertabular}
\usepackage{ifpdf}
\usepackage{enumerate}
\usepackage{fancyhdr}
\usepackage[top=1.5in, bottom=1.5in, left=1.5in, right=1.5in]{geometry}

\allowdisplaybreaks

\newcommand{\beq}{\begin{equation}}
\newcommand{\eeq}{\end{equation}}
\newcommand{\beqa}{\begin{eqnarray}}
\newcommand{\eeqa}{\end{eqnarray}}
\newcommand{\beaa}{\begin{eqnarray*}}
\newcommand{\ben}{\begin{eqnarray*}}
\newcommand{\eaa}{\end{eqnarray*}}
\newcommand{\een}{\end{eqnarray*}}

\newcommand \nc {\newcommand}

\newtheorem{theorem}{Theorem}[section]
\newtheorem{lemma}[theorem]{Lemma}

\newtheorem{corollary}[theorem]{Corollary}
\newtheorem{definition}[theorem]{Definition}

\nc \thref{Theorem \ref}
\nc \leref{Lemma \ref}
\nc \prref{Proposition \ref}
\nc \coref{Corollary \ref}
\nc \deref{Definition \ref}
\nc \exref{Example \ref}
\nc \reref{Remark \ref}

\newcommand{\A}{\mathcal{A}}
\newcommand{\B}{\mathcal{B}}
\newcommand{\C}{\mathbb{C}}
\newcommand{\D}{\mathcal{D}}

\renewcommand{\H}{\mathcal{H}}

\renewcommand{\L}{\mathcal{L}}
\newcommand{\M}{\mathcal{M}}

\renewcommand{\O}{\mathcal{O}}

\newcommand{\T}{\mathcal{T}}

\newcommand{\Z}{\mathbb{Z}}

\newcommand{\f}{\mathbf{f}}

\newcommand{\q}{\mathbf{q}}
\renewcommand{\t}{\mathbf{t}}

\def\dim{\mathop{\rm dim}\nolimits}

\def\d{\partial}

\def\({\left(}
\def\){\right)}
\def\[{\left[}
\def\]{\right]}
\def\<{\left\langle}
\def\>{\right\rangle}

\def\one{{\bf 1}}

\def\D{{\mathcal D}}
\def\la{\lambda}

\def\al{\alpha}
\def\de{\delta}
\def\be{\beta}
\def\Si{\Sigma}

\newcommand{\leftexp}[2]{{\vphantom{#2}}^{#1}{#2}}

\title[Eynard--Orantin recursion for the ancestors]
{The Eynard--Orantin recursion for the total ancestor potential}
\author{Todor Milanov}
\address{Kavli IPMU (WPI) \\ The University of Tokyo \\ Kashiwa \\ Chiba 277-8583 \\ Japan}
\email{todor.milanov@ipmu.jp}

\begin{document}

\begin{abstract}
It was proved recently that the correlation functions of a semi-simple 
cohomological field theory satisfy the so called Eynard--Orantin
topological recursion. We prove that in the settings of singularity
theory, the relations can be expressed in terms of periods integrals
and the so called phase forms. In particular, we prove that the
Eynard-Orantin recursion is equivalent to
$N$ copies of Virasoro constraints for the ancestor potential,
which follow easily from the definition of the potential. 
\end{abstract}
\maketitle
\tableofcontents
\addtocontents{toc}{\protect\setcounter{tocdepth}{1}}

\section{Introduction}

The Eynard--Orantin recursion (see \cite{EO}) was discovered first for the correlation functions of certain matrix integrals.
However, its applications go beyond the theory of matrix models. The recursion is turning into a powerful tool
for computing the correlation functions in various quantum field
theories. In particular, it provides  an efficient algorithm for
computing quite complicated invariants such as Gromov--Witten
invariants and certain polynomial invariants in knot theory. 

In order to set up the recursion one needs an analytic curve, called {\em
  spectral curve}, two holomorphic functions on it, and a certain
symmetric 2-form satisfying some additional properties. At first, one
might think that this is a serious restriction, so the applications
would be only limited. The surprising fact however, is that the
initial data can be determined from the
1-point and the 2-point correlation functions only (of a given quantum
field theory) (see \cite{DMSS}). This feature makes the recursion
quite universal. In particular, this observation was exploited in the
paper \cite{BOSS}, where the authors prove that the ancestor
Gromov--Witten (GW for shortness) invariants of manifolds with a
semi-simple quantum cohomology can be computed via the Eynard--Orantin
recursion. Although our work appears after \cite{BOSS}, the main observation namely, that one should study the
$n$-point series \eqref{n-point-series} and that they should satisfy the 
Eynard--Orantin recursion with kernel given by formulas
\eqref{kernel-1} and \eqref{kernel-2}  was done independently.

The goal in this paper is to interpret the Eynard--Orantin recursion
in terms of differential operator constraints for the total ancestor
potential. In particular, this allows us to obtain a simple proof of
the recursion relation. In particular, we prove that the
correlation functions can be expressed in terms of period integrals
and phase forms, which suggests that they should be compared to the
correlation functions of the twisted Vertex algebra representation
introduced in \cite{BM}.

\subsection{Preliminary notation}\label{sec:notation}
Let $f\in \O_{\C^{2l+1},0}$ be the germ of a holomorphic function with an
isolated critical point at $0$. We fix a miniversal deformation
$F(t,x),$ $t\in B$ and a primitive
form $\omega$ in the sense of K. Saito \cite{S1, MS}, so that $B$ inherits a Frobenius structure (see \cite{He,SaT}). In particular we have the following identifications (c.f. Section \ref{sec:frobenius}):
\ben
T^*B\cong TB\cong B\times T_0B\cong B\times H,
\een
where $H$ is the Jacobi algebra of $f$, the first isomorphism is given by the residue pairing, the
second by the flat residue metric and the last one is the
Kodaira--Spencer isomorphism 
\beq\label{KS}
T_0B \cong H,\quad \d/\d t_i\mapsto \left. \d_{t_i}F\right|_{t=0} \ {\rm mod}\ 
 ( f_{x_0},\dots,f_{x_{2l}}).
\eeq
We need also the period integrals
\beq\label{period}
I^{(k)}_\alpha(t,\lambda) = - d^B\ (2\pi)^{-l}\, \d_\la^{k+l} \
\int_{\alpha_{t,\lambda}} d^{-1}\omega\ \in T_t^*B\cong H,
\eeq
where $\alpha$ is a cycle from the vanishing cohomology, $d^B$ is the de Rham differential on $B$, and $d^{-1}\omega$ is any $(n-1)$-form $\eta$
such that $d\eta=\omega$. The periods are multivalued analytic
functions on $B\times \mathbb{P}^1$  with poles along the so called {\em discriminant locus} (c.f. Section \ref{sec:periods}). 
We make use of the following series
\ben
\f^\alpha(t,\la;z) = \sum_{k\in \Z} \ I^{(k)}_\alpha(t,\la)\,
(-z)^k,\quad
\phi^\alpha(t,\lambda;z) = \sum_{k\in \Z} \
I^{(k+1)}_\alpha(t,\la)(-z)^k\, d\lambda.
\een
Note that $\phi^\alpha(t,\lambda;z) =d^{\mathbb{P}^1} \f^\alpha(t,\la;z)$.

Let $B_{ss}\subset B$ be the subset of semi-simple points, i.e.,
points $t\in B$ such that the critical values of $F(t,\cdot)$ form a
coordinate system in a neighborhood of $t$.  For every $t\in B_{ss}$
Givental's higher-genus reconstruction formalism gives rise to
ancestor correlation functions of the following form
\beq\label{correlators}
\langle v_1 \psi_1^{k_1},\dots,v_n\psi^{k_n}\rangle_{g,n}(t),\quad
v_i\in H,\quad k_i\in \Z_+ (1\leq i\leq n).
\eeq
Apriory, each correlator depends analytically on $t\in B_{ss}$, but it might have poles along the divisor $B\setminus{B_{ss} }$. 
Given $n$ vanishing cycles $\alpha_1,\dots,\alpha_n$ and a generic
point $t\in B$ we define the following $n$-point symmetric forms
\beq\label{n-point-series}
\omega_{g,n}^{\alpha_1,\dots,\alpha_n}(t;\la_1,\dots,\la_n) = 
{\Big\langle} \phi^{\alpha_1}_+(t,\la_1,\psi_1),\dots,  \phi^{\alpha_n}_+(t,\la_n,\psi_n)\Big\rangle_{g,n}(t),
\eeq
where the $+$ means truncation of the terms in the series with
negative powers of $z$. The functions \eqref{n-point-series} will be
called {\em $n$-point series} of genus $g$. They should be
interpreted formally via their Laurent series expansions at the
singular points. They are the main object of our study and we expect that
they have many remarkable properties yet to be discovered. Probably
the first question to be addressed is whether the $n$-point functions are
global objects, i.e., multivalued analytic functions with finite order
poles on the {\em configuration space} $\mathcal{C}_n(\mathbb{P}^1) =
\mathcal{F}_n(\mathbb{P}^1) /\mathfrak{S}_n$, where  
\ben
\mathcal{F}_n(\mathbb{P}^1) = \Big\{ (\lambda_1,\dots,\lambda_n)\in
(\mathbb{P}^1)^n\ :\ \lambda_i\neq \la_j \mbox{ for }i\neq j\Big\}
\een
and $\mathfrak{S}_n$ is the symmetric group acting by permutation of
the coordinates. 

\subsection{Statement of the results}
Let us assume that $t\in B_{ss}$ is generic so that the function
$F(t,\cdot)$ has $N:={\rm dim}_\C \, H$ pairwise different critical
values $u_j(t).$ Let $\be_j$ be a cycle vanishing over $\la=u_j$. We
introduce the following quadratic differential operator
\beq\label{vir-j}
Y_{t,\la}^{u_j} = :(\d_\lambda\,\widehat{\f^{\beta_j}}(t,\lambda))^2: \ +\  P^{\beta_j,\beta_j}_0(t,\lambda),
\eeq
where $:\  :$ is the normal ordering, the differential operator
(c.f. Section \ref{ssympl})
\ben
\d_\lambda\,\widehat{\f^{\beta_j}}(t,\lambda)=\d_\lambda\,\widehat{\f_+^{\beta_j}}(t,\lambda)+\d_\lambda\,\widehat{\f_-^{\beta_j}}(t,\lambda)
\een
is defined by
\begin{align}\label{quant1}
\d_\la\widehat{\f^{\beta_j}}(t,\la)_+ &= \sum_{k=0}^\infty \sum_{i=1}^N \, (-1)^{k+1} \, (I_{\beta_j}^{(k+1)}(t,\la), v^i) \,
\hbar^{1/2} \, \frac{\d}{\d q_k^i} \,,
\\ \label{quant2}
\d_\la\widehat{\f^{\beta_j}}(t,\la)_- &= \sum_{k=0}^\infty \sum_{i=1}^N \, (I_{\beta_j}^{(-k)}(t,\la), v_i) \,
\hbar^{-1/2} \, q_k^i  \,,
\end{align}
where $\{v^i\}$ and $\{v_i\}$ are dual bases for $H$ with respect to
the residue pairing $(,)$. Finally, $P_0^{\beta_j,\beta_j}$ is the
free term  in the Laurent series expansion of the {\em propagator}
\ben
[\d_\la\widehat{\f_+^{\beta_j}}(t,\mu),\d_\la\widehat{\f_-^{\beta_j}}(t,\la)]=2\,(\mu-\la)^{-2}
+ \sum_{k=0}^\infty P_k^{\beta_j,\beta_j}(t,\la)\, (\mu-\la)^k.
\een
The definition \eqref{vir-j} is very natural from the point of view of vertex
algebras (see \cite{BM}). It is the field
that determines a representation of the Virasoro vertex operator algebra
with central charge 1 on the {\em Fock space} 
\ben
\C_\hbar[[q_0,q_1+{\bf 1},q_2,\dots]],\quad q_k=(q_k^1,\dots,q_k^N),
\quad \C_\hbar:=\C(\!(\sqrt{\hbar})\!),
\een
where $\one$ is the unity of the local algebra $H$. We may assume that
$v_1=\one$.

Let us denote by $\A_t(\hbar;\q)$ the {\em total ancestor potential}
of the singularity. By definition, it is a vector in the Fock space of
the form (c.f. Section \ref{asop})
\ben
\exp\Big(\sum_{g,n=0}^\infty \frac{1}{n!}\,{\Big\langle}
\t(\psi_1),\dots,\t(\psi_n){\Big\rangle}_{g,n}(t)\ \hbar^{g-1}\Big),
\een  
where $\t(\psi)=\sum_{k,i} t_k^i \, v_i\psi^k$ and the relation
between the set of formal variables $\{q_k^i\}$ and $\{t_k^i\}$ is
given by the {\em dilaton shift}
\beq\label{dilaton-shift}
t_k^i=
\begin{cases}
q_k^i & \mbox{ if } (k,i)\neq (1,1),\\
q_1^1+1 & \mbox{ otherwise}.
\end{cases}
\eeq 
We define the following set of differential operators 
\beq\label{eo-dop}
L_{m-1,i}=\frac{1}{4}\,\sum_{j=1}^N\, {\rm Res}_{\la=u_j}\, \frac{
  (I^{(-m-1)}_{\beta_j}(t,\la),v_i) \, Y_{t,\la}^{u_j}
}{(I^{(-1)}_{\beta_j}(t,\la),\one) }\, d\la,\quad m\geq 0,\quad 1\leq
i\leq N.
\eeq
Note that although the periods are multi-valued analytic
functions, the above expression is single valued with respect to the
local monodromy around $\la=u_j$, so the residue is well defined. 
Our first result is the following .
\begin{theorem}\label{t1}
The total ancestor potential satisfies the following constraints:
\ben
L_{m-1,i}\,\A_t(\hbar;\q) = 0,\quad m\geq 0,\quad 1\leq i\leq N.
\een
\end{theorem}
The proof follows easily from the definition of the
ancestor potential and some known properties of the periods
\eqref{period}. More precisely one can prove that
$Y_{t,\la}^{u_j}\,\A_t$ is regular near $\la=u_j$, so each residue
vanishes. The main property of the construction \eqref{eo-dop} 
is that $L_{m-1,i}$ has only one term that
involves the dilaton-shifted variable $q_1^1$ and this term is
$q_1^1\d/\d q_m^i$. This fact allows us to interpret the differential
operator constraints as a system of recursion relations. Our next
result is the following.
\begin{theorem}\label{t2}
The differential operator constraints determine a system of
recursion relations that coincide with the local Eynard--Orantin recursion. 
\end{theorem} 
We postpone the definition of the Eynard--Orantin recursion until Section
\ref{sec:EO}.
Following \cite{BOSS} we give the definition of the recursion only
locally. It will be very important to find the global
formulation (see \cite{BE}) as well, but for now this seems to be a
very challenging problem, except may be for simple and simple elliptic
singularities in singularity theory or the projective line and its
orbifold versions in GW theory. The problems is that in 
the settings of singularity theory, or GW theory, the spectral curve
has highly transcendental nature. Its description requires inverting the period map: very
classical and very difficult problem. 
Our set up is slightly different from the standard conventions (\cite{BE,BOSS,DMSS,EO}),
because we would like to work with multivalued correlation
functions. The main idea is that whatever is the spectral curve $\Sigma$,
we always have a projection $\Sigma\to \mathbb{P}^1$ which in general
is an
infinite sheet branched covering. Galois theory tells us that studying
the field of meromorphic functions on $\Si$ is the same as the field
of multivalued meromorphic functions on $\mathbb{P}^1$ invariant under
some monodromy (Galois) group. Our proposal is to formulate the
Eynard--Orantin recursion for correlation functions on $\mathbb{P}^1$
that take values in some local system $\L$. The definition that we
give in Section \ref{sec:EO} is just a first attempt to set up this
idea. Probably a better formulation is possible if one takes into
account more examples not only the ones that come from GW theory. 

Let us point out that although we work in the settings of singularity
theory, the differential operators \eqref{eo-dop} can be defined for
any semi-simple Frobenius manifold that has an Euler vector field. The
period vectors should be introduced as the solutions to a system of
differential equations (see Lemma \ref{lem:periods}) and it is easy to
see that the proofs of Theorems \ref{t1} and \ref{t2} remain the
same. In particular, we obtain the main result of \cite{BOSS}
\begin{corollary}\label{c1}
The $n$-point series of a semi-simple cohomological field theory
satisfy the local  Eynard--Orantin recursion relations. 
\end{corollary}
\subsection{Acknowledgement}
I am thankful to M. Mulase for teaching me how the
Eynard--Orantin recursion can be constructed in terms of 1- and
2-point functions.  Also, I thank S. Shadrin for useful
e-mail communication and clarifying some points from
\cite{BOSS}. This work is supported by Grant-In-Aid and by the World
Premier International Research Center Initiative (WPI Initiative),
MEXT, Japan. 

\section{Frobenius structures in singularity theory}\label{sec:sing}

Let $f\colon(\C^{2l+1},0)\rightarrow (\C,0)$ be the germ of a holomorphic function with an isolated critical point of multiplicity $N$. Denote by 
\begin{equation*}
H = \C[[x_0,\ldots,x_{2l}]]/(\d_{x_0}f,\ldots,\d_{x_{2l}}f)
\end{equation*}
the \emph{local algebra} of the critical point; then $\dim H=N$. 

\begin{definition}\label{dmvdef}
A \emph{miniversal deformation} of $f$ is a germ of a holomorphic function $F\colon(\C^N\times \C^{2l+1},0)\to (\C,0)$ satisfying the following two properties:
\begin{enumerate}
\item[(1)]
$F$ is a deformation of $f$, i.e., $F(0,x)=f(x)$.
\item[(2)]
The partial derivatives $\d F/\d t^i$ $(1\leq i\leq N)$ project to a basis in the local algebra 
$$
\O_{\C^N,0}[[x_0,\dots,x_{2l}]]/\langle \d_{x_0}F,\dots,\d_{x_{2l}}F\rangle.
$$
\end{enumerate}
Here we denote by $t=(t^1,\dots,t^N)$ and $x=(x_0,\dots,x_{2l})$ the standard coordinates on $\C^N$ and $\C^{2l+1}$ respectively, and $\O_{\C^N,0}$ is the algebra of germs at $0$ of holomorphic functions on $\C^N.$
\end{definition}

We fix a representative of the holomorphic germ $F$, which we denote again by $F$, with a domain $X$ constructed as follows. Let 
\begin{equation*}
B_\rho^{2l+1}\subset \C^{2l+1} \,, \qquad 
B=B_\eta^N\subset \C^N \,, \qquad 
B_\delta^1\subset \C 
\end{equation*}
be balls with centers at $0$ and radii $\rho,\eta$, and $\delta$, respectively.
We set 
\begin{equation*}
S=B\times B_\delta^1 \subset\C^N\times\C \,, \quad 
X=(B\times B_\rho^{2l+1})\cap \phi^{-1}(S) \subset\C^N\times\C^{2l+1} \,,
\end{equation*}
where
\ben
\phi\colon B\times B_\rho^{2l+1}\to B\times\C \,,
\qquad (t,x)\mapsto (t,F(t,x)) \,.
\een 
This map induces a map $\phi\colon X\to S$ and we denote by $X_s$ or $X_{t,\la}$ the fiber 
\ben
X_s = X_{t,\la} = \{(t,x)\in X \,|\, F(t,x)=\la\} \,,\qquad s=(t,\la)\in S.
\een  
The number $\rho$ is chosen so small that for all $r$, $0<r\leq \rho$, the fiber $X_{0,0}$ intersects transversely the boundary $\d B_r^{2l+1}$ of the ball with radius $r$. Then we choose the numbers $\eta$ and $\delta$ small enough so that for all $s\in S$ the fiber $X_s$ intersects transversely the boundary $\d B_\rho^{2l+1}.$ Finally, we can assume without loss of generality that the critical values of $F$ are contained in a disk $B_{\delta_0}^1$ with radius $\delta_0<1<\delta$.

Let $\Si$ be the {\em discriminant} of the map $\phi$, i.e., the set
of all points $s\in S$ such that the fiber $X_s$ is singular. Put 
\begin{equation*}
S'=S\setminus{\Si} \subset\C^N\times\C \,, \qquad 
X'=\phi^{-1}(S') \subset X \subset\C^N\times\C^{2l+1} \,.
\end{equation*}
Then the map $\phi\colon X'\to S'$ is a smooth fibration, called 
the \emph{Milnor fibration}. In particular, all smooth fibers are diffeomorphic to $X_{0,1}$.
The middle homology group of the smooth fiber, equipped with the bilinear form
$(\cdot|\cdot)$ equal to $(-1)^l$ times the intersection form, 
is known as the \emph{Milnor lattice} $Q=H_{2l}(X_{0,1};\Z)$. 

For a generic point $s\in\Si$, the singularity of the fiber $X_s$
is Morse. Thus, every choice of a path from $(0,1)$ to $s$ avoiding $\Si$
leads to a group homomorphism $Q \to H_{2l}(X_s;\Z)$. The kernel of this
homomorphism is a free $\Z$-module of rank $1$. A generator  
$\al\in Q$ of the kernel is called a \emph{vanishing cycle} if 
$(\al|\al) = 2$. 

\subsection{Frobenius structure}\label{sec:frobenius}

Let $\T_B$ be the sheaf of holomorphic vector fields on $B$. Condition (2) in \deref{dmvdef} implies that the map 
$$
\d/\d{t^i}\mapsto \d F/\d t^i \mod \langle \d_{x_0} F,\dots,\d_{x_{2l}}F\rangle \qquad (1\leq i\leq N)
$$ 
induces an isomorphism between $\T_B$ and $p_*\O_C$, where $p\colon X\to B$ is the natural projection $(t,x)\mapsto t$ and 
\ben
\O_C:=\O_X/\langle \d_{x_0} F,\dots,\d_{x_{2l}}F\rangle
\een
is the structure sheaf of the critical set of $F$. In particular, since $\O_C$ is an algebra, the sheaf $\T_B$ is equipped with an associative commutative multiplication, which will be denoted by $\bullet.$ It induces a product $\bullet_t$ on the tangent space of every point $t\in B$. The class of the function $F$ in $\O_C$ defines a vector field $E\in \T_B$, called the {\em Euler vector field}. 

Given a holomorphic volume form $\omega$ on $(\C^{2l+1},0)$, possibly
depending on $t\in B$, we can equip $p_*\O_C$ with the so-called
\emph{residue pairing}:
\ben
(\psi_1(t,x),\psi_2(t,x)) :=
\Big(\frac{1}{2\pi i}\Big)^{2l+1}\int_{\Gamma_\epsilon} 
\frac{\psi_1(t,y)\ \psi_2(t,y)}
{\d_{y_0} F \cdots \d_{y_{2l}} F } \,\omega\,,
\een
where $y=(y_0,\dots,y_{2l})$ is a unimodular coordinate system for
$\omega$ (i.e. $\omega=dy_0\wedge \cdots \wedge dy_{2l}$) and the integration cycle $\Gamma_\epsilon$ is supported on 
$|\d_{y_0} F|= \cdots =|\d_{y_{2l}}F|=\epsilon$. 
Using that $\T_B\cong p_*\O_C$, we get a non-degenerate complex
bilinear form $(\ ,\ )$ on $\T_B$, which we still call residue
pairing.    

For $t\in B$ and $z\in\C^*$, let
$\B_{t,z}$ be a semi-infinite cycle in $\C^{2l+1}$ of the following type:
\ben
\B_{t,z}\in \lim_{\rho \to \infty} \, H_{2l+1}(\C^{2l+1},\{ 
  \mathrm{Re}\, z^{-1} F(t,x)<-\rho\} ;\C) \cong \C^N \,.
\een
The above homology groups form a vector bundle on $B\times \C^*$ equipped
naturally with a Gauss--Manin connection, and $\B=\B_{t,z}$ may be viewed as a
flat section. According to K.\ Saito's theory of {\em primitive forms} \cite{S1,MS}
there exists a form $\omega$, called primitive, such that the oscillatory
integrals ($d^B$ is the de Rham differential on $B$)
\ben
J_\B(t,z):= (2\pi z)^{-l-\frac{1}{2}}\ (zd^B)\, 
\int_{\B_{t,z}} e^{z^{-1}F(t,x)}\omega \in \T_B^*
\een
are horizontal sections for the following connection: 
\beqa\label{frob_eq3}
\nabla_{\d/\d t^i} & = &  \nabla^{\rm L.C.}_{\d/\d t^i} - z^{-1}(\d_{t^i} \bullet_t),
\qquad 1\leq i\leq N \\
\label{frob_eq4}
\nabla_{\d/\d z} & = &  \d_z - z^{-1} \theta + z^{-2} E\bullet_t \,.
\eeqa
Here $\nabla^{\rm L.C.}$ is the Levi--Civita connection associated with the residue pairing and
\ben
\theta:=\nabla^{\rm L.C.}E-\Big(1-\frac{d}{2}\Big){\rm Id},
\een 
where $d$ is some complex number. 
In particular, this means that the residue pairing and the multiplication $\bullet$ form a {\em Frobenius structure} on $B$
of conformal dimension $d$ with identity $1$ and Euler vector field $E$. For the definition of a Frobenius structure we refer to \cite{Du} .

Assume that a primitive form $\omega$ is chosen. Note that the flatness of the Gauss--Manin connection implies that the residue pairing is flat. Denote by $(\tau^1,\dots, \tau^N)$ a coordinate system on $B$ that is flat with respect to the residue metric, and write $\partial_i$ for the vector field $\partial/\partial{\tau^i}$. We can further modify the flat coordinate system so that the Euler field is the sum of a constant and linear fields: 
\ben
E=\sum_{i=1}^N (1-d_i) \tau^i \partial_{i} + \sum_{i=1}^N \rho_i \partial_i \,.
\een
The constant part represents the class of $f$ in $H$, and the spectrum
of degrees $d_1,\dots, d_N$ ranges from $0$ to $d.$ 
Note that in the flat coordinates $\tau^i$ the operator $\theta$ (called sometimes the \emph{Hodge grading operator}) assumes diagonal form:
\ben
\theta(\d_i) = \Bigl(\frac{d}{2}-d_i\Bigr) \d_i \,, \qquad\quad 1\leq i\leq N \,.
\een
Finally, the vectors $v_i\in H$ appearing in formula
\eqref{correlators} are the images of the flat vector fields $\d_i$
via the Kodaira--Spencer isomorphism \eqref{KS}.

\subsection{Period integrals}\label{sec:periods}
Given a middle homology class $\al\in H_{2l}(X_{0,1};\C)$, we denote by $\al_{t,\la}$ its parallel transport to 
the Milnor fiber $X_{t,\la}$. Let $d^{-1}\omega$ be any $2l$-form whose 
differential is $\omega$. We can integrate $d^{-1}\omega$ over $\al_{t,\la}$
and obtain multivalued functions of $\la$ and $t$ 
ramified around the discriminant in $S$ (over which 
the Milnor fibers become singular).

\begin{definition}\label{dperiods}
To $\al\in H_{2l}(X_{0,1};\C)$, we associate the {\em 
period vectors} $I^{(k)}_\al(t,\la)\in H\ (k\in \Z)$ defined by
\beq\label{periods} 
(I^{(k)}_\al(t,\la), \partial_i):= 
-(2\pi)^{-l} \d_\la^{l+k} \d_i \int_{\al_{t,\la}} d^{-1}\omega \,,
\qquad 1\leq i\leq N \,.
\eeq 
\end{definition}

Note that this definition is consistent with the operation of stabilization of
singularities. Namely, adding the squares of two new variables does not change
the right-hand side, since it is offset by an extra differentiation 
$(2\pi)^{-1}\partial_{\la}$. In particular, this defines the period
vector for a negative value of $k\geq -l$ with $l$ as large as one wishes.
Note that, by definition, we have 
\ben
\d_\la I^{(k)}_\al(t,\la) = I^{(k+1)}_\al(t,\la) \,, \qquad k\in\Z\,.
\een
The following lemma is due to A.\ Givental \cite{G3}.
\begin{lemma}\label{lem:periods}
The period vectors \eqref{periods} satisfy the differential 
equations
\begin{align}\label{periods:de1}
\d_i I^{(k)} &= -\d_i\bullet_t(\d_\la I^{(k)})\,, \qquad\quad 1\leq i\leq N \,, 
\\ \label{periods:de2}
(\la-E\bullet_t) \d_{\la}I^{(k)} &= \Bigl(\theta-k-\frac12\Bigr) I^{(k)} \,.
\end{align}
\end{lemma}

Using equation \eqref{periods:de2},
we analytically extend the period vectors to all 
$|\la|>\delta$. It follows from \eqref{periods:de1} that the period vectors 
have the symmetry
\beq\label{tinv}
I^{(k)}_\al(t,\la)\ = \ I^{(k)}_\al(t-\la\one,0) \,,
\eeq  
where $t \mapsto t-\la\one $ denotes the time-$\la$ translation 
in the direction of the flat vector field $\one $ obtained from $1\in H$. 
(The latter represents identity elements for all the products
$\bullet_t$.)

\subsection{Stationary phase asymptotic}\label{sec:asymptotic}

Let $u_i(t)$ ($1\leq i\leq N$) be the critical values of $F(t,\cdot)$. For a generic $t$, they form a local coordinate system on $B$ in which the Frobenius multiplication and the residue pairing are diagonal. Namely,
\ben
\d/\d u_i \, \bullet_t\, \d/\d u_j = \delta_{ij}\d/\d u_j \,,\quad 
\(\d/\d u_i,\d/\d u_j \) = \delta_{ij}/\Delta_i \,,
\een
where $\Delta_i$ is the Hessian of $F$ with respect to the volume form $\omega$ at the critical point corresponding to the critical value $u_i.$ 
Therefore, the Frobenius structure is \emph{semi-simple}.

We denote by $\Psi_t$ the following linear isomorphism
\ben
\Psi_t\colon \C^N\rightarrow T_t B \,,\qquad 
e_i\mapsto \sqrt{\Delta_i}\d/\d u_i \,,
\een
where $\{e_1,\dots,e_N\}$ is the standard basis for $\C^N$.
Let $U_t$ be the diagonal matrix with entries $u_1(t),\ldots, u_N(t)$. 

According to Givental \cite{G1}, the system of differential equations (cf.\ \eqref{frob_eq3}, \eqref{frob_eq4})
\begin{align}\label{de_1}
z\d_i J(t,z) &= \d_i\bullet_t J(t,z) \,,\qquad\quad 1\leq i\leq N \,,
\\
\label{de_2}
z\d_z J(t,z) &= (\theta-z^{-1} E\bullet_t) J(t,z)
\end{align}
has a unique formal asymptotic solution of the form $\Psi_t R_t(z) e^{U_t/z}$, 
where 
\ben
R_t(z)=1+R_1(t)z+R_2(t)z^2+\cdots \,,
\een
and $R_k(t)$ are linear operators on $\C^N$ uniquely determined from the differential 
equations \eqref{de_1} and \eqref{de_2}.  
Introduce the formal series
\beq\label{falpha}
\f_\al(t,\la;z) = \sum_{k\in \Z} I^{(k)}_\al(t,\la) \, (-z)^k \,.
\eeq
Note that for $A_1$-singularity $F(t,x)=x^2/2+t$ we have $u:=u_1(t)=t$
and the series \eqref{falpha} takes the form
\ben
\f_{A_1}(t,\la;z) = \sum_{k\in \Z}\, I^{(k)}_{A_1}(u,\la)\, (-z)^k,
\een
where
\ben
\begin{aligned}
I^{(k)}_{A_1}(u,\la) & = (-1)^k\, \frac{(2k-1)!!}{2^{k-1/2}}\,
(\la-u)^{-k-1/2},\quad k\geq 0 \\
I^{(-k-1)}_{A_1}(u,\la) & = 2\, \frac{2^{k+1/2}}{(2k+1)!!}\,
(\la-u)^{k+1/2}, \quad k\geq 0.
\end{aligned}
\een
The key lemma, which is due to Givental \cite{G3} is the following.
\begin{lemma}\label{vanishing_a1}
Let\/ $t\in B$ be generic and\/ $\be$ be a vanishing cycle vanishing over the point\/ $(t,u_i(t))\in \Si$. Then for all\/ $\la$ near\/ $u_i:=u_i(t)$, we have
\ben
\f_{\be}(t,\la;z) = \Psi_t R_t(z)\,  e_i\,  \f_{A_1}(u_i,\la;z)\,.
\een
\end{lemma}

\section{Symplectic loop space formalism}\label{sec4}

The goal of this section is to introduce Givental's quantization
formalism (see \cite{G2}) and use it to define the higher genus potentials in
singularity theory.

\subsection{Symplectic structure and quantization}\label{ssympl}

The space $\H:=H(\!(z^{-1})\!)$ of formal Laurent series in $z^{-1}$ with
coefficients in $H$ is equipped with the following \emph{symplectic form}: 
\ben
\Omega(\phi_1,\phi_2):={\rm Res}_z \(\phi_1(-z),\phi_2(z)\) \,,
\qquad \phi_1,\phi_2\in\H \,,
\een 
where, as before, $(,)$ denotes the residue pairing on $H$
and the formal residue ${\rm Res}_z$ gives the coefficient in front of $z^{-1}$.

Let $\{\d_i\}_{i=1}^{ N}$ and $\{\d^i\}_{i=1}^N$ be dual bases of $H$ with respect to the residue pairing.
Then
\ben
\Omega(\d^i(-z)^{-k-1}, \d_j z^l) = \de_{ij} \de_{kl} \,.
\een
Hence, a Darboux coordinate system is provided by the linear functions $q_k^i$, $p_{k,i}$ on $\H$ given by:
\ben
q_k^i = \Omega(\d^i(-z)^{-k-1}, \cdot) \,, \qquad
p_{k,i} = \Omega(\cdot, \d_i z^k) \,.
\een
In other words,
\ben
\phi(z) = \sum_{k=0}^\infty \sum_{i=1}^N q_k^i(\phi) \d_i z^k +  
\sum_{k=0}^\infty \sum_{i=1}^N p_{k,i}(\phi) \d^i(-z)^{-k-1} \,,
\qquad \phi\in\H \,.
\een  
The first of the above sums will be denoted $\phi(z)_+$ and the second $\phi(z)_-$.

The \emph{quantization} of linear functions on $\H$ is given by the rules:
\ben
\widehat q_k^i = \hbar^{-1/2} q_k^i \,, \qquad
\widehat p_{k,i} = \hbar^{1/2} \frac{\d}{\d q_k^i} \,.
\een
Here and further, $\hbar$ is a formal variable. We will denote by $\C_\hbar$ the field
$\C(\!(\hbar^{1/2})\!)$.

Every $\phi(z)\in\H$ gives rise to the linear function $\Omega(\phi,\cdot)$ on $\H$,
so we can define the quantization $\widehat\phi$. Explicitly,
\beq\label{phihat}
\widehat\phi = -\hbar^{1/2} \sum_{k=0}^\infty \sum_{i=1}^N q_k^i(\phi) \frac{\d}{\d q_k^i}
+ \hbar^{-1/2} \sum_{k=0}^\infty \sum_{i=1}^N p_{k,i}(\phi) q_k^i \,.
\eeq
The above formula makes sense also for $\phi(z)\in H[[z,z^{-1}]]$ if we interpret $\widehat\phi$
as a formal differential operator in the variables $q_k^i$ with coefficients in $\C_\hbar$.

\begin{lemma}\label{lphihat}
For all\/ $\phi_1,\phi_2\in\H$, we have\/ $[\widehat\phi_1, \widehat\phi_2] = \Omega(\phi_1,\phi_2)$.
\end{lemma}
\begin{proof}
It is enough to check this for the basis vectors $\d^i(-z)^{-k-1}$, $\d_i z^k$, in which case it is true by definition.
\end{proof}
It is known that the operator series
$
\mathcal{R}_t(z):=\Psi_t R_t(z)\Psi_t^{-1}
$
is a symplectic transformation. Moreover, it has the form $e^{A(z)},$ where $A(z)$ is an infinitesimal symplectic transformation. 
A linear operator $A(z)$ on $\H:=H(\!(z^{-1})\!)$ is infinitesimal symplectic if and only if the map $\phi\in \H \mapsto A\phi\in \H$ is a Hamiltonian vector field with a Hamiltonian given by the quadratic function $h_A(\phi) = \frac{1}{2}\Omega(A\phi,\phi)$. 
By definition, the \emph{quantization} of $e^{A(z)}$ is given by the differential operator $e^{\widehat{h}_A},$ where the quadratic Hamiltonians are quantized according to the following rules:
\ben
(p_{k,i}p_{l,j})\sphat = \hbar\frac{\d^2}{\d q_k^i\d q_l^j} \,,\quad 
(p_{k,i}q_l^j)\sphat = (q_l^jp_{k,i})\sphat = q_l^j\frac{\d}{\d q_k^i} \,,\quad
(q_k^iq_l^j)\sphat = \frac1{\hbar} q_k^iq_l^j \, .
\een     

\subsection{The total ancestor potential }\label{asop}
Let us make the following convention. Given a vector 
\ben
\q(z) = \sum_{k=0}^\infty q_k z^k \in H[z] \,, \qquad
q_k=\sum_{i=1}^N q_k^i\d_i \in H \,,
\een
its coefficients give rise to a vector sequence  $q_0,q_1,\dots$.
By definition, a {\em formal  function} on $H[z]$,
defined in the formal neighborhood of a given point $c(z)\in H[z]$, is a formal power
series in $q_0-c_0, q_1-c_1,\dots$. Note that every operator acting on
$H[z]$ continuously in the appropriate formal sense induces an
operator acting on formal functions.

The \emph{Witten--Kontsevich tau-function} is the following generating series:
\beq\label{D:pt}
\D_{\rm pt}(\hbar;Q(z))=\exp\Big( \sum_{g,n}\frac{1}{n!}\hbar^{g-1}\int_{\overline{\M}_{g,n}}\prod_{i=1}^n (Q(\psi_i)+\psi_i)\Big),
\eeq
where $Q_0,Q_1,\ldots$ are formal variables, and $\psi_i$ ($1\leq
i\leq n$) are the first Chern classes of the cotangent line bundles on
$\overline{\M}_{g,n}$ (see \cite{W1,Ko1}).
It is interpreted as a formal
function of $Q(z)=\sum_{k=0}^\infty Q_k z^k\in \C[z]$, defined in the
formal neighborhood of $-z$. In other words, $\D_{\rm pt}$ is a formal power series
in $Q_0,Q_1+1,Q_2,Q_3,\dots$ with coefficients in $\C(\!(\hbar)\!)$.

Let $t\in B$ be a {\em semi-simple} point, so that the critical values
$u_i(t)$ ($1\leq i\leq N$) of $F(t,\cdot)$ form a coordinate
system. Recall also the flat coordinates
$\tau=(\tau^1(t),\dots,\tau^N(t))$ of $t$. The {\em total ancestor
  potential} of the singularity is defined as follows
\beq\label{ancestor}
\A_t(\hbar; \q(z)) = \widehat{\mathcal{R}}_t\ \prod_{i=1}^N\, \D_{\rm pt}(\hbar\Delta_i;\leftexp{i}{\q}(z)) \in \C_\hbar[[q_0,q_1+\one,q_2\dots]],
\eeq
where 
$
\mathcal{R}_t(z):=\Psi_t R_t(z)\Psi_t^{-1}
$
and 
\ben
\leftexp{i}{\q}(z) = \sum_{k=0}^\infty \sum_{j=1}^N q_k^j(\d_j u_i) z^k \,.
\een
\subsection{Proof of Theorem \ref{t1}}
Using Lemma \ref{vanishing_a1}, it is easy to see that the ratio
\ben
 (I^{(-m-1)}_{\beta_j}(t,\la),v_i) / (I^{(-1)}_{\beta_j}(t,\la),\one) 
\een
is analytic in a neighborhood of $\la=u_j$ for all $m\geq 0$. Furthermore, 
$
Y^{u_j}_{t,\la} \, \widehat{\mathcal{R}}_t  = \widehat{\mathcal{R}}_t \,
Y^{A_1}_{u_j},  
$
where $Y^{A_1}_{u_j}$ is the differential operator \eqref{vir-j} on
the variables
$Q_k^j:=\leftexp{j}{q_k}/\sqrt{\Delta_j}\ (k\geq 0)$, defined for the
$A_1$-singularity $F=x^2/2 + u_j$ (see Lemma 6.12 in \cite{BM}). Using
that the Virasoro operators in the Virasoro constraints for the point
coincide with the polar part of $Y^{A_1}_{u_j}$ (see Section 8.3 in
\cite{BM}), we get that $Y^{u_j}_{t,\la} \, \A_t(\hbar;\q)$ is a
formal series in $\q$ whose coefficients are analytic at
$\la=u_j$. This implies that the residue at $\la=u_j$ vanishes, which
completes the proof. 
\qed 

\subsection{The Virasoro recursions}
Let us compute the coefficient in front of $q_1^1$ in $L_{m-1,i}$. The
contribution from the $j$-th residue is computed as follows. We put
$\beta=\beta_j$ to avoid cumbersome notation. By
definition the differential operator $Y_{t,\la}^{u_j}$ is a sum of
quadratic expressions of the following 3 types:
\ben
(-1)^{k'+k''}\, (I^{(k'+1)}_\beta(t,\la),v^a)\,
(I^{(k''+1)}_\beta(t,\la),v^b)\, \hbar \d_{q_{k'}^a}\d_{q_{k''}^b}\, ,
\een
\beq\label{type-2}
2(-1)^{k''+1}\, (I^{(-k')}_\beta(t,\la),v_a)\,
(I^{(k''+1)}_\beta(t,\la),v^b)\, q_{k'}^a \d_{q_{k''}^b}\, ,
\eeq
and
\ben
(I^{(-k')}_\beta(t,\la),v_a)\,
(I^{(-k'')}_\beta(t,\la),v_b)\ \hbar^{-1} q_{k'}^a\,q_{k''}^b\, ,
\een
where the sum is over all $k',k''\geq 0$ and $a,b=1,2,...,N.$ The only
contribution could come from the terms \eqref{type-2}. The coefficient
in front of $q_1^1 \d_{q_k^b}$ is 
\beq\label{coef}
\frac{1}{2} \, (-1)^{k+1}\, {\rm Res}_{\la=u_j} \ (I^{(-m-1)}_\beta(t,\la),v_i)\, (I^{(k+1)}_\beta(t,\la),v^b).
\eeq
\begin{lemma}\label{hrp}
The following identity holds
\ben
\sum_{j=1}^N {\rm Res}_{\la=u_j} \ (I^{(k')}_\beta(t,\la),v_a)\,
(I^{(k'')}_\beta(t,\la),v^b) \, d\la=2(-1)^{k'}\delta_{a,b}\delta_{k'+k'',0},
\een
for all $k',k''\in \Z$ and $a,b=1,2,\dots,N.$
\end{lemma}
\proof
According to Lemma \ref{vanishing_a1} we have
\ben
I^{(k)}_\be(t,\la) = \sum_{l=0}^\infty R_l\, (-\d_\la)^{-l} \, I_{A_1}^{(k)}(u_j,\la)e_j.
\een
Using this identity we find that the $j$-th term in the above sum is 
\beq\label{contr-j}
\sum_{l',l''=0}^\infty (\leftexp{T}{R}_{l'}v_a,e_j)
(\leftexp{T}{R}_{l''}v^b,e_j)(-1)^{l'+l''} 
{\rm Res}_{\la=u_j}\,
I_{A_1}^{(k'+l')}(u_j,\la) I_{A_1}^{(k''+l'')}(u_j,\la)d\la.
\eeq
The above residue is non-zero only if $k'+l'=-k''-l''$. In the latter case using
integration by parts we find that the residue is
\ben
(-1)^{k'+l'} {\rm Res}_{\la=u_j}\,
I_{A_1}^{(0)}(u_j,\la) I_{A_1}^{(0)}(u_j,\la)d\la = 2(-1)^{k'+l'}.
\een
The sum \eqref{contr-j} becomes
\ben
2(-1)^{k'}\sum_{l',l''=0}^\infty (\leftexp{T}{R}_{l'}v_a,e_j)
(\leftexp{T}{R}_{l''}v^b,e_j)(-1)^{l''} \delta_{l'+l'',-k'-k''} 
\een
If we sum over all $j=1,2,\dots,N$, since $\{e_j\}$ is an orthonormal
basis of $H$, we get 
\ben
2(-1)^{k'}\sum_{l',l''=0}^\infty
(\leftexp{T}{R}_{l'}v_a,\leftexp{T}{R}_{l''}v^b)(-1)^{l''}
\delta_{l'+l'',-k'-k''} .
\een
Using the symplectic condition $R(z)\leftexp{T}{R}(-z)=1$ we see that
the only non-zero contribution in the above sum comes from the terms
with $l'=l''=0$, which completes the proof.
\qed  

The above Lemma implies that the coefficient \eqref{coef} is non-zero
only if $k=m$ and $b=i$ and in the latter case the coefficient is
1. In order to obtain a recursion relation for the correlators
\eqref{correlators} we replace $q_1^1=t_1^1-1$ and compare the
genus $g$ degree $n$ (with respect to $\t$) terms in the identity
\beq\label{recursion}
\sum_{m=0}^\infty
(-1)^{m+1}\,(I^{(m+1)}_{\beta}(t,\mu),v^a)L_{m-1,a}\, \A_t(\hbar;\q)=0.
\eeq
Note that if we ignore the dilaton shift, then $(Y_{t,\la}^{u_j} \A_t
d\la\cdot d\la)/\A_t$ (here $\cdot$ is the {\em symmetric} product of
differential forms) is a sum of terms of five different types. The
first two are
\beq\label{loop:stable}
\frac{\hbar^{g-1}}{n!}{\Big\langle} \phi^{\beta_j}_+(t,\la;\psi_1),\phi^{\beta_j}_+(t,\la;\psi_2),\t,\dots,\t{\Big\rangle}_{g-1,n+2},
\eeq
and
\beq\label{split:stable}
\sum_{\substack{ g'+g''=g\\ n'+n''=n }} \frac{\hbar^{g-1}}{ (n')! (n'')! } 
{\Big\langle} \phi^{\beta_j}_+(t,\la;\psi_1),\t,\dots,\t{\Big\rangle}_{g',n'+1}\ 
{\Big\langle} \phi^{\beta_j}_+(t,\la;\psi_1),\t,\dots,\t{\Big\rangle}_{g'',n''+1}.
\eeq
The other three types are
\beq\label{loop:unstable}
P_0^{\beta_j,\beta_j}(t,\la),
\eeq
\beq\label{split:unstable1}
\frac{2\hbar^{g-1}}{n!}\Omega(\phi^{\beta_j}_-(t,\la;z),\t(z))\, 
{\Big\langle} \phi^{\beta_j}_+(t,\la;\psi_1),\t,\dots,\t{\Big\rangle}_{g,n},
\eeq
and
\beq\label{split:unstable2}
\hbar^{-1}\Omega(\phi^{\beta_j}_-(t,\la;z),\t(z))\ \Omega(\phi^{\beta_j}_-(t,\la;z),\t(z)).
\eeq
Let us point out that the ancestor correlators \eqref{correlators} are
{\em tame} (see \cite{G3}), which by definition means that they 
vanish if $k_1+\cdots +k_n>3g-3+n$. In particular, the ancestor
potential does not have non-zero correlators in the genus-0 unstable range
$(g,n)=(0,0),(0,1),$ and $(0,2)$. However, motivated by the above formulas, it is convenient
to extend the definition in the unstable range as well by setting 
\ben
\begin{aligned}
{\Big\langle} \phi^{\beta_j}_+(t,\la;\psi_1),\t {\Big\rangle}_{0,2}
& := 
\Omega(\phi^{\beta_j}_-(t,\la;z),\t(z))\\
{\Big\langle} \phi^{\beta_j}_+(t,\la;\psi_1),
\phi^{\beta_j}_+(t,\la;\psi_1) {\Big\rangle}_{0,2}
& :=  P_0^{\beta_j,\beta_j}(t,\la)
\end{aligned}
\een
and keeping the remaining unstable genus-0 correlators $0$. If we allow
such unstable correlators; then the terms \eqref{split:unstable1} and
\eqref{split:unstable2} become the unstable part of the sum
\eqref{split:stable}, while \eqref{loop:unstable} becomes the unstable
correlator in the set \eqref{loop:stable}. 

The above discussion and the fact that the dilaton shift changes
$L_{m-1,a}$ simply by an additional differentiation $-\d/\d t_m^a$
yields the following identities:
\beqa\label{rec:1}
&&
{\Big\langle}
\phi^{\beta}_+(t,\mu;\psi_1),\t,\dots,\t{\Big\rangle}_{g,n+1} = \\
\label{rec:2}
&&
\frac{1}{4}\,\sum_{j=1}^N\, {\rm Res}_{\la=u_j}\,
\frac{ [\widehat{\phi^\beta}_+(t,\mu),\widehat{\f^{\beta_j}}_-(t,\la)] }
 { (I^{(-1)}_{\beta_j}(t,\la),\one)\,d\la } \times 
\left(
{\Big\langle}\phi^{\beta_j}_+(t,\la;\psi_1),\phi^{\beta_j}_+(t,\la;\psi_2),\t,\dots,\t{\Big\rangle}_{g-1,n+2}+
\phantom{\sum_{\substack{ g'+g''=g\\ n'+n''=n }} }\right.
\\
\label{rec:3}
&&
\left.
\sum_{\substack{ g'+g''=g\\ n'+n''=n }} {n\choose n'}\ 
{\Big\langle} \phi^{\beta_j}_+(t,\la;\psi_1),\t,\dots,\t{\Big\rangle}_{g',n'+1}\ 
{\Big\langle} \phi^{\beta_j}_+(t,\la;\psi_1),\t,\dots,\t{\Big\rangle}_{g'',n''+1}
\right),
\eeqa
where $(g,n)$ is assumed to be in the stable range, i.e., $2g-2+n>0$,
we are allowing unstable correlators on the RHS, and we suppressed the dependence
of the correlators on $t\in B_{ss}$. Note that the RHS involves
differential forms that should be treated formally with respect to the
symmetric product $\cdot$ of differential forms. 
\ben
\frac{d\mu\cdot d\la\cdot d\la }{d\la} = d\mu\cdot d\la = d\la\cdot d\mu.
\een  
The residue contracts $d\la$, so at the end the RHS involves only
$d\mu$. Let us point out that the tameness condition is
crucial, because it implies that for the correlator insertion of
the type $\sum_m I^{(m+1)}_\beta(t,\la)\,(-\psi)^m$, only finitely
many terms contribute. In particular, although the infinite sum of
differential operators in \eqref{recursion} does not make sense in
general, our argument goes through since on each step only finitely
many terms of the sum \eqref{recursion} contribute.

\section{The Eynard--Orantin recursion}\label{sec:EO}

The initial data for setting up the local Eynard--Orantin
recursion is a complex line $\C$ with $N$ marked points
$u_1,\dots,u_N$ and a certain set of 1- and 2-forms defined only
locally.  More precisely, for each $i$ we have a 
multi-valued holomorphic 1-form $\omega^i(\la)=P^i(\la)\,d\la$ defined in
some disk-neighborhood $D_i$  of $u_i$, s.t.,
\ben
P^i(\lambda) := \sum_{k=0}^\infty\  P_k^i\ (\la-u_i)^{k+1/2} \, .
\een 
For each pair $(i,j)$ we have a symmetric 2-form
$\omega^{ij}(\mu,\la):=P^{ij}(\mu,\la)\, d\la\cdot d\mu$ on $D_i\times
D_j$ obeying the symmetry $(i,\mu)\leftrightarrow (j,\la)$ and such
that the function
\ben
(\mu-u_i)^{1/2}(\la-u_j)^{1/2} P^{ij}(\mu,\la) = (\mu-u_i)^{1/2}(\la-u_j)^{1/2} P^{ji}(\la,\mu) 
\een
is analytic
on $D_i\times D_j$ except for a pole (with no residues) of order 2
along the diagonal in the case when $i=j$. In the latter case, we assume that the
differentials are normalized in such a way that the Laurent series
expansion with respect to $\mu$ in the annulus $0<|\la-\mu|<|\la-u_i|$
has the form 
\beq\label{ope}
P^{ii}(\mu,\la) = \frac{2}{(\mu-\la)^2} +
\sum_{k=0}^\infty\ P_k^{ii}(\la)\ (\mu-\la)^k.
\eeq
Note that for each fixed $k\geq 0$, the functions $P^{ii}_k(\la)$ are holomorphic on the punctured disk
$D_i^*:=D_i\setminus{\{u_i\}}$ with a finite order pole at $\la=u_i$. 

Let us denote by $\L_i$ the local system in a neighborhood of $D_i$
defined by the multi-valued function $(\la-u_i)^{1/2}$. We define a
system of symmetric multi-valued analytic differential forms
$\omega_{g,n}^{\alpha_{1},\dots,\alpha_{n}}(\la_1,\dots,\la_n)$ for
$\alpha_{k}\in \L_{i_k}$ and $\la_k\in D^*_{i_k}$ that are compatible
with the (local) monodromy action on the local systems, i.e., the
analytic continuation along a small loop around $\la_k=u_{i_k}$
transforms the differential form into
$\omega_{g,n}^{\alpha_{1},\dots,\sigma(\alpha_{k}),\dots,\alpha_{n}}(\la_1,\dots,\la_n)$,
where $\sigma$ is the corresponding monodromy action on $\L_{i_k}$. 

Let $\alpha\in \L_i,\beta\in \L_j$ be any sections; then the
base of the recursion is the  following
\ben
\begin{aligned}
\omega_{0,1}^{\alpha}(\la) &  = 0, \\
\omega_{0,2}^{\alpha,\beta}(\mu,\la)& =
\begin{cases}
P^{ij}(\mu,\la)\, d\mu\cdot d\la & \mbox{ if } (i,\mu)\neq (j,\la), \\ 
P^{ii}_{0}(\la) \, d\la\cdot d\la& \mbox{ if } (i,\mu)=(j,\la).
\end{cases}
\end{aligned}
\een
The {\em kernel} of the recursion is the following ratio of 1 forms:
\beq\label{kernel}
K^{\alpha,\beta}(\mu,\la) = \frac{1}{2}\, \frac{\oint_{C_\la}
  P^{ij}(\mu,\la')\,d\la' }{ P^j(\la) }\ \frac{d\mu}{d\la},
\eeq
where we fix $\mu\in D_i\setminus{\{u_i\}}$ and select a simple loop
$C_{\la}$ in $D_j$ based at $\la$ that goes around $u_j$. Then the 
recursion takes the form
\ben
\omega_{g,n+1}^{\alpha_0,\alpha_1,\dots,\alpha_n}(\la_0,\la_1,\dots,\la_n)
= \sum_{j=1}^N \, {\rm Res}_{\la=u_j}\,
K^{\alpha_0,\beta_j}(\la_0,\la)\times  \\
\left( 
\omega_{g-1,n+2}^{\beta_j,\beta_j,\alpha_1,\dots,\alpha_n}(\la,\la,\la_1,\dots,\la_n)
+
\sum_{\substack{g'+g''=g \\ I'\sqcup I''=\{1,\dots,n\}} }
\omega_{g',n'+1}^{\beta_j,\alpha_{I'} }(\la,\la_{I'}) \, \omega_{g'',n''+1}^{\beta_j,\alpha_{I''} }(\la,\la_{I''}) \right),
\een
where we are assuming that $2g-2+n>0$, $\beta_j\in \L_j$, the sum in the big
brackets is over all splittings, $n'$ and $n''$ are the number of
elements respectively in $I'$ and $I''$, and for a subset $I\subset
\{1,2,\dots,n\}$ we adopt the standard multi-index notation $x_I=(x_{i_1},\dots,x_{i_k})$.
Note that although the functions are multivalued, the local monodromy
about $\la=u_j$ leaves the expression invariant, so the residue is
well defined. 
 
\subsection{Proof of Theorem \ref{t2}}
In the settings of singularity theory, for a generic $t\in B_{ss}$ we
let the marked points be the critical values $u_i=u_i(t)$. The
choice of a section of the local system $\L_i$ is the same as choosing
a vanishing cycle over $\la=u_i$. Let
$\omega_{g,n}^{\beta_1,\dots,\beta_n}(\la_1,\dots,\la_n)$ be the
$n$-point series \eqref{n-point-series}. 

In order to prove that these forms satisfy the Eynard--Orantin
recursion, it is enough to notice that 
\ben
\frac{1}{\sqrt{\hbar}} \, [\widehat{\phi^{\beta}}_+(t,\la),\t(\psi)]= \phi^\beta_+(t,\la;\psi).
\een
Applying this formula $n$ times to \eqref{rec:1}--\eqref{rec:3} with
$\beta=\beta_1,\dots,\beta_n$ we obtain the Eynard--Orantin
recursion with 
\beq\label{kernel-1}
\omega^{ij}(\mu,\la) = [\widehat{\phi^\alpha}_+(t,\mu),\widehat{\phi^\beta}_-(t,\la)], 
\eeq
and 
\beq\label{kernel-2}
P^j(\la) = 4(I^{(-1)}_\beta(t,\la),\one),
\eeq
where $\alpha,\beta$ are vanishing cycles vanishing respectively over
$\mu=u_i$ and $\la=u_j$. Note that 
\ben
\oint_{C_\la} P^{ij}(\mu,\la') d\la'\cdot d\mu= 2\  [\widehat{\phi^\alpha}_+(t,\mu),\widehat{\f^\beta}_-(t,\la)]
\een
so the kernel is given by formula \eqref{kernel}.

In the opposite direction, in order to prove that the Eynard--Orantin
recursion implies the Virasoro constraints, it is enough to notice
that according to Lemma \ref{hrp} we have the following identity
\ben
\t(\psi) = \frac{1}{2}\, \sum_{j=1}^N \ {\rm
  Res}_{\la=u_j}\ \Omega(\f^{\beta_j}_-(t,\la;z), \t(z)) \,
\phi^{\beta_j}(t,\la;\psi)\ . 
\qed
\een
Finally, let us point out that using Lemma \ref{vanishing_a1} one can
express the Laurent series expansions of $\omega^{ij}(\mu,\la)$ and
$\omega^j(\la)$ in terms of the symplectic operator series
$\mathcal{R}.$ 
The answer is the following. Let $V_{kl}\in {\rm End}(H)$ be defined via
\ben
\sum_{k,l=0}^\infty V_{kl} w^k z^l = \frac{
  1-\leftexp{T}{\mathcal{R}}(-w)\mathcal{R}(-z) }{z+w} ,
\een
then $(\mu-u_i)^{1/2}(\la-u_j)^{1/2}\, P^{ij} (\mu,\la)$ has the
following Taylor's series expansion 
\ben
\frac{\delta_{ij} }{(\mu-\la)^2} \, (\mu-u_i+\la-u_j) +
\sum_{k,l=0}^\infty
2^{k+l+1}(e_i,V_{kl}e_j)\frac{(\mu-u_i)^k}{(2k-1)!!}\, \frac{(\la-u_j)^l}{(2l-1)!!}
\een
 Note that if $i=j$ and  we fix $\la$ near $u_i$; then the Laurent series
 expansion of $P^{ij}(\mu,\la)$ about $\mu=\la$ does take the form
 \eqref{ope}.
The Taylor's series expansion of $P^j(\la)$ is
\ben
8\, \sum_{k=0}^\infty \frac{ (-1)^k 2^{k+1/2} }{ (2k+1)!! }\,
(\mathcal{R}_k\,e_j,\one)\, (\la-u_j)^{k+1/2} .
\een
Up to an appropriate normalization of the correlation functions, our
answer agrees with the formulas in \cite{BOSS}.


\begin{thebibliography}{JKV2}

\bibitem{AGV}{V.~Arnold, S.~Gusein-Zade, and A.~Varchenko.}
{\em Singularities of Differentiable maps.} Vol. II. Monodromy and
  Asymptotics of Integrals. Boston, MA: Birkh\"auser Boston,
  1988. viii+492 pp

\bibitem{BM}{B.~Bakalov and T.~Milanov}
\emph{W-constraints for the total descendant potential of a simple singularity.}
Preprint (2012);
http:/\!/arxiv.org/abs/1203.3414

\bibitem{BE}{V.~Bouchard and B.~Eynard}
\emph{Think globally, compute locally.}
Preprint (2012);
http:/\!/arxiv.org/abs/1211.2302


\bibitem{BOSS}{P.~Dunnin--Barkowski, N.~Orantin, S.~Shadrin, and  L.~Spitz.}
{\em Identification of the Givental formula with the spectral curve topological recursion procedure.} 
Preprint (2012);
http:/\!/arxiv.org/abs/1211.4021

\bibitem{EO}{B.~Eynard and N.~Orantin}
\emph{Invariants of algebraic curves and topological expansion.} 
Comm. in Number Theory and Physics \textbf{1}(2007): 347--552

\bibitem{Du}
B.~Dubrovin,
\textit{Geometry of 2D topological field theories}. 
In: ``Integrable systems and quantum groups'' 
(Montecatini Terme, 1993), 120--348, Lecture Notes
in Math., 1620, Springer, Berlin, 1996.

\bibitem{DMSS}{O.~Dumitrescu, M.~Mulase, B.~Safnuk, and A.~Sorkin.}
\emph{
The spectral curve of the Eynard-Orantin recursion via the Laplace transform.}
Preprint (2012);
http:/\!/arxiv.org/abs/1202.1159

\bibitem{G1} A.~Givental.
{\em Semisimple Frobenius structures at higher genus.}
Internat. Math. Res. Notices 2001, no. 23, 1265-1286

\bibitem{G2}
A.~Givental.
\textit{Gromov--Witten invariants and quantization of quadratic Hamiltonians}. 
Mosc. Math. J. \textbf{1} (2001), 551--568

\bibitem{G3}{A.~Givental.}
\emph{$A_{n-1}$ singularities and $n$KdV Hierarchies.} Mosc. Math. J. 3.2(2003): 475--505

\bibitem{He}{C.~Hertling.}
\emph{Frobenius Manifolds and Moduli Spaces for Singularities.}
Cambridge Tracts in Mathematics, 151. Cambridge University Press, Cambridge, 2002. x+270 pp

\bibitem{Ko1}
M.~Kontsevich, 
\textit{Intersection theory on the moduli space of curves and the matrix Airy function}. 
Comm. Math. Phys. \textbf{147} (1992), 1--23.

\bibitem{S1}{K.~Saito.}
{\em On Periods of Primitive Integrals, I.}
Preprint RIMS(1982)

\bibitem{SaT}{K.~Saito and A.~Takahashi.}
{\em From primitive forms to Frobenius manifolds.}
From Hodge theory to integrability and TQFT tt*-geometry, 31-48, Proc. Sympos. Pure Math., 78, Amer. Math. Soc., Providence, RI, 2008

\bibitem{MS} 
M.~Saito, 
\textit{On the structure of Brieskorn lattice}. 
Ann. Inst. Fourier \textbf{39} (1989), 27--72. 

\bibitem{W1}
E.~Witten,
\textit{Two-dimensional gravity and intersection theory on moduli space}. 
In: ``Surveys in differential geometry,'' 243--310, 
Lehigh Univ., Bethlehem, PA, 1991. 

\end{thebibliography}
\end{document}